\documentclass[11pt,reqno]{amsart}

\usepackage{amsmath,amssymb,amsfonts,bm}
\usepackage{stmaryrd}
\usepackage{bbold}
\usepackage{times}
\usepackage[all]{xy}
\usepackage{mathrsfs}
\usepackage{graphicx}

\usepackage{amsthm}

\theoremstyle{plain}
\newtheorem{thm}{Theorem}
\newtheorem{lem}[thm]{Lemma}
\newtheorem{cor}[thm]{Corollary}
\newtheorem{prop}[thm]{Proposition}

\theoremstyle{definition}

\newtheorem{example}[thm]{Example}

\theoremstyle{remark}

\renewcommand{\Im}{\operatorname{Im}}
\newcommand{\Hom}{\operatorname{Hom}}
\newcommand{\End}{\operatorname{End}}
\newcommand{\Aut}{\operatorname{Aut}}

\newcommand{\ann}{\operatorname{ann}}

\newcommand{\sS}{\mathcal{S}}

\DeclareMathOperator{\ad}{ad}
\DeclareMathOperator{\reg}{reg}

\newcommand{\edim}{\operatorname{Dim}_\e}

\newcommand{\Spec}{\operatorname{Spec}}

\newcommand{\Pic}{\operatorname{Pic}}

\newcommand{\Id}{\operatorname{Id}}
\newcommand{\rank}{\operatorname{rank}}

\newcommand{\tr}{\operatorname{Tr}}
\newcommand{\trace}{\operatorname{trace}}

\newcommand{\ch}{\operatorname{char}}

\newcommand{\into}{\hookrightarrow}
\newcommand{\tto}{\longrightarrow}
\newcommand{\iso}{\stackrel{\sim}{\tto}}

\renewcommand{\bar}[1]{\overline{#1}}

\newcommand{\Mat}{\operatorname{M}}

\DeclareMathOperator{\Irr}{Irr}

\renewcommand{\phi}{\varphi}

\newcommand{\cO}{\mathcal{O}}

\newcommand{\bl}{\boldsymbol{l}}
\newcommand{\br}{\boldsymbol{r}}
\newcommand{\bb}{\boldsymbol{b}}

\newcommand{\fp}{\mathfrak{p}}

\newcommand{\e}{\varepsilon}

\renewcommand{\k}{\mathbb{k}}
\newcommand{\FF}{\mathbb{F}}

\newcommand{\ZZ}{\mathbb{Z}}
\newcommand{\QQ}{\mathbb{Q}}
\newcommand{\RR}{\mathbb{R}}

\newcommand{\cen}{\mathcal{Z}}
\newcommand{\bil}{\operatorname{Bil}}

\newcommand{\Hint}{\textstyle\int}
\newcommand{\Cocom}{\operatorname{C}}

\newdir{ >}{{}*!/-5pt/\dir{>}}

\renewcommand{\labelenumi}{(\alph{enumi})}

\begin{document}

\title[Frobenius algebras]%
{Some applications of Frobenius algebras to Hopf algebras}

\author{Martin Lorenz}

\address{Department of Mathematics, Temple University,
    Philadelphia, PA 19122}

\email{lorenz@temple.edu}

\urladdr{www.math.temple.edu/~lorenz}


\thanks{Research of the author supported in part by NSA Grant H98230-09-1-0026}



\begin{abstract}
This expository article presents a unified ring theoretic approach, based on the theory of Frobenius algebras, 
to a variety of results on Hopf algebras. These include a theorem of S. Zhu on the degrees
of irreducible representations, the so-called class equation, the
determination of the semisimplicity locus of the Grothendieck ring, the spectrum of the
adjoint class and a non-vanishing result
for the adjoint character.
\end{abstract}

\maketitle



\section*{Introduction}

\subsection{} 
It is a well-known fact that all finite-dimensional Hopf algebras over a field are Frobenius 
algebras. 
More generally, working over a commutative base ring $R$ with trivial
Picard group, any Hopf $R$-algebra that
is finitely generated projective over $R$ is a Frobenius $R$-algebra \cite{bP71}.
This article explores the Frobenius property, and some consequences thereof, for Hopf algebras and for 
certain algebras that 
are closely related to Hopf algebras without generally being Hopf algebras themselves: the Grothendieck ring
$G_0(H)$ of a split semisimple Hopf algebra $H$ and the representation algebra $R(H) \subseteq H^*$.
Our principal goal is to quickly derive various consequences from the fact that the latter algebras are 
Frobenius, or even symmetric, thereby giving a unified ring theoretic approach to a variety of known results
on Hopf algebras. 

\subsection{}
The first part of this article, consisting of four sections, is entirely devoted to Frobenius and symmetric algebras
over commutative rings; its sole purpose is to 
deploy the requisite ring theoretical tools. 
The content of these sections is classical over fields and the case of general commutative base rings
is easily derived along the same lines. In the interest of readability,  we have opted 
for a self-contained development.
The technical core of this part are the construction of certain central idempotents in
Propositions~\ref{P:idempotents} and \ref{P:idempotents2}
and the description of the separability locus of
a Frobenius algebra in Proposition~\ref{P:separability}. The essence of the latter
proposition goes back to D.~G.~Higman \cite{dH55}. 

\subsection{}
Applications to Hopf algebras are given in Part 2. We start by considering 
Hopf algebras that are finitely generated projective over a commutative ring.
After reviewing some standard facts,
due to Larson-Sweedler \cite{rLmS69}, Pareigis \cite{bP71}, and
Oberst-Schneider  \cite{uOhS73}, 
concerning the Frobenius property of Hopf algebras, we
spell out the content of the aforementioned Propositions~\ref{P:idempotents} and \ref{P:idempotents2}
in this context (Proposition~\ref{P:Hopf7}).
A generalization, due to Rumynin \cite{dRxx}, of Frobenius' classical theorem
on the degrees of irreducible complex representations of finite groups
follows in a few lines from this result (Corollary~\ref{C:HopfIrreducible}). 

The second section of Part 2 focuses on semisimple Hopf algebras
$H$ over fields, specifically their Grothendieck rings $G_0(H)$.
As an application of Proposition~\ref{P:Hopf7}, we derive a result of S. Zhu \cite{sZ93}
on the degrees of certain irreducible representations of $H$ (Theorem~\ref{T:Zhu}).
The celebrated class equation for semisimple Hopf algebras is
presented as an application of Proposition~\ref{P:idempotents} in Theorem~\ref{T:classeqn}. 
The proof given here follows the outline of our earlier proof in \cite{mL98}, with
a clearer separation of the purely ring theoretical underpinnings. Other applications
concern a new proof of an integrality result, originally due to Sommerh\"auser \cite{yS98},
for the eigenvalues of the ``adjoint class'' (Proposition~\ref{P:adjoint}), the determination of the
semisimplicity locus of $G_0(H)$ (Proposition~\ref{P:semisimple}), and a non-vanishing
result for the adjoint character (Proposition~\ref{P:group-like}).

For the sake of simplicity, we have limited ourselves for the most part to base fields of 
characteristic $0$. In some cases, this restriction can be removed with the aid of
$p$-modular systems. For example, as has been observed by Etingof and Gelaki \cite{pEsG98},
any bi-semisimple Hopf algebra over an algebraically closed field of positive characteristic,
along with all its irreducible representations,
can be lifted to characteristic $0$. We plan to address this aspect more fully in a sequel
to this article.

In closing I wish to thank the referee for a careful reading of this article and for 
valuable comments.


\part{RING THEORY} \label{1}


Throughout this first part of the paper, $R$ will denote a commutative ring and $A$ will
be an $R$-algebra that is finitely generated projective over $R$.


\section{Preliminaries on Frobenius and symmetric algebras} \label{S:preliminaries}



\subsection{Definitions} \label{SS:definition}

\subsubsection{Frobenius algebras} \label{SSS:basic}

Put $A^\vee = \Hom_R(A,R)$; this is an $(A,A)$-bimo\-dule via
\begin{equation} \label{E:VeeAction}
(afb)(x) = f(bxa) \qquad (a,b,x \in A, f \in A^\vee)\ .
\end{equation}
The algebra $A$ is called   \emph{Frobenius} 
if $A \cong A^\vee$ as left $A$-modules. 
This is equivalent to $A \cong A^\vee$ as right $A$-modules. Indeed, 
using the standard isomorphism
$A_A \iso (A^{\vee\vee})_A = \Hom_R({}_AA^\vee,R)_A$ given by  $a \mapsto (f \mapsto f(a))$, 
one deduces from ${}_AA^\vee \cong {}_AA$ that 
\begin{equation} \label{E:chain}
A_A \cong \Hom_R({}_AA^\vee,R)_A
\cong \Hom_R({}_AA,R)_A = {A^\vee\!}_A \ .
\end{equation}
The converse is analogous. 

More precisely, any isomorphism
$L \colon {}_AA \iso {}_AA^\vee$ has the form
\begin{equation*}
L(a) = a\lambda \qquad (a \in A)\ ,
\end{equation*}
where we have put $\lambda = L(1) \in A^\vee$. The linear form $\lambda$ is called a
\emph{Frobenius homomorphism}. Tracing $1 \in A$ through \eqref{E:chain} 
we obtain $1 \mapsto (a\lambda \mapsto (a\lambda)(1) = \lambda(a))
\mapsto (a \mapsto \lambda(a)) = \lambda$. Hence, the resulting isomorphism
$R \colon A_A \iso {A^\vee\!}_A$ is explicitly given by
\begin{equation*}
R(a) = \lambda a \qquad (a \in A)\ .
\end{equation*}
In particular, $\lambda$ is also a free generator of $A^\vee$ as right $A$-module.
The automorphism $\alpha \in \Aut_\text{$R$-alg}(A)$ that is given by 
$\lambda a = \alpha(a)\lambda$ for $a \in A$ is called the \emph{Nakayama 
automorphism}.

\subsubsection{Symmetric algebras}

If  $A \cong A^\vee$ as $(A,A)$-bimodules then the algebra $A$ is called 
\emph{symmetric}. 
Any $(A,A)$-bimodule
isomorphism $A \cong A^\vee$ restricts to an isomorphism
of Hochschild cohomology modules $H^0(A,A) \cong H^0(A,A^\vee)$.
Here, $H^0(A,A) = \cen(A)$ is the center of $A$, and
$H^0(A,A^\vee)= A^{\vee}_{\trace}$ consists of all
\emph{trace forms} on $A$, that is, $R$-linear forms $f \in A^\vee$
vanishing on all Lie commutators $[x,y] = xy-yx$ for $x,y \in A$. Thus, if $A$ is symmetric then we 
obtain an isomorphism of $\cen(A)$-modules
\begin{equation} \label{E:center}
\cen(A) \iso A^{\vee}_{\trace} \ .
\end{equation}

\subsection{Bilinear forms} \label{SS:nonsingular}

\subsubsection{Nonsingularity} 
Let $\bil(A;R)$ denote the $R$-module consisting of all 
$R$-bilinear forms $\beta \colon A \times A \to R$. 
Putting $\br_\beta(a) =  \beta(a,\,.\,)$ for $a \in A$,
we obtain an isomorphism
of $R$-modules
\begin{equation} \label{E:Frob1}
\br \colon \bil(A;R) \iso \Hom_R(A,A^\vee)\ ,\quad \beta \mapsto \br_\beta \ .
\end{equation}
The bilinear form $\beta$ is called left nonsingular if $\br_\beta$ is an 
isomorphism. Inasmuch as $A$ and $A^\vee$ are locally isomorphic projectives
over $R$, it suffices to assume that
$\br_\beta$ is surjective; see \cite[Cor. 4.4(a)]{dE95}.

Similarly, there is an isomorphism
\begin{equation} \label{E:Frob2}
\bl \colon \bil(A;R) \iso \Hom_R(A,A^\vee)\ ,\quad \beta \mapsto \bl_\beta
\end{equation}
with $\bl_\beta(a) = \beta(\,.\,,a)$,
and $\beta$ is called right nonsingular if $\bl_\beta$ is an 
isomorphism. Right and left nonsingularity are in fact equivalent. 
After localization, this follows from \cite[Prop. XIII.6.1]{sL02}. 
In the following, we will therefore call
such forms simply \emph{nonsingular}.

\subsubsection{Dual bases} \label{SSS:dual}
Fix a nonsingular bilinear form $\beta \colon A \times A \to R$. Since $A$ is finitely generated projective over 
$R$, we have a canonical isomorphism $\End_R(A) \cong A \otimes_R A^\vee$; see \cite[II.4.2]{nB70}.
Thus, the isomorphism 
$\bl_\beta$ in \eqref{E:Frob2} yields an isomorphism $\End_R(A) \iso A \otimes _R A$.
Writing the image of $\Id_A \in \End_R(A)$ as $\sum_{i=1}^n x_i \otimes y_i \in A \otimes_R A$, we have
\begin{equation} \label{E:Frob3}
a = \sum_i \beta(a,y_i) x_i \quad\text{for all $a \in A$}.
\end{equation}
Conversely, assume that $\beta \colon A \times A \to R$ is such that there are elements 
$\{ x_i \}_1^n,  \{ y_i \}_1^n \subseteq A$ satisfying \eqref{E:Frob3}. Then
any $f \in A^\vee$ satisfies 
\begin{equation} \label{E:Frob4}
f = \sum_i \bl_\beta(y_i)f(x_i) = \bl_\beta(\sum_i  f(x_i) y_i)\ .
\end{equation}
This shows that $\bl_\beta \colon A \to A^\vee$ is
surjective, and hence $\beta$ is nonsingular.  To summarize, a
bilinear form $\beta \in \bil(A;R)$ is nonsingular if and only if there exist ``dual bases''
$\{ x_i \}_1^n,  \{ y_i \}_1^n \subseteq A$ satisfying \eqref{E:Frob3}. 

Note also that, for a nonsingular $R$-bilinear form $\beta \colon A \times A \to R$ and 
elements $\{ x_i \}_1^n,  \{ y_i \}_1^n \subseteq A$, condition \eqref{E:Frob3} is equivalent to
\begin{equation} \label{E:Frob3'}
b = \sum_i \beta(x_i,b) y_i \quad\text{for all $b \in A$}.
\end{equation}
Indeed, both \eqref{E:Frob3} and \eqref{E:Frob3'} are equivalent to
$\beta(a,b) = \sum_i \beta(a,y_i) \beta(x_i,b)$ for all $a,b \in A$.

\subsubsection{Associative bilinear forms} 
An $R$-bilinear form $\beta  \colon A \times A \to R$ is called \emph{associative} if  $\beta(xy,z) = \beta(x,yz)$  
for all $x,y,z \in A$. Let $\bil_\text{assoc}(A;R)$ denote the $R$-submodule 
of $\bil(A;R)$ consisting of 
all such forms. Under the isomorphism  \eqref{E:Frob1}, $\bil_\text{assoc}(A;R)$
corresponds to the submodule $\Hom(A_A,{A^\vee}_A) \subseteq \Hom_R(A,A^\vee)$.
Similarly, \eqref{E:Frob2} yields an isomorphism of $R$-modules 
$\bil_\text{assoc}(A;R) \cong\Hom({}_AA,{}_AA^\vee)$. Therefore:
\begin{quote}
The algebra $A$ is Frobenius if and only if there exists a nonsingular associative 
$R$-bilinear form $\beta \colon A \times A \to R$. 
\end{quote}

\subsubsection{Symmetric forms} \label{SSS:symmetric}
The form $\beta$ is called \emph{symmetric} if $\beta(x,y) = \beta(y,x)$ for all
$x,y \in A$. The isomorphisms $\br$ and $\bl$ in  \eqref{E:Frob1}, \eqref{E:Frob2} agree
on the submodule consisting of all symmetric bilinear forms, and they yield an isomorphism
between the $R$-module consisting of all associative symmetric bilinear forms on $A$
and the submodule $\Hom({}_AA_A,{}_A{A^\vee}_A) \subseteq \Hom_R(A,A^\vee)$ consisting of all 
$(A,A)$-bimodule maps $A \to A^\vee$. Thus:
\begin{quote}
The algebra $A$ is symmetric if and only if there exists a nonsingular associative symmetric 
$R$-bilinear form $\beta \colon A \times A \to R$. 
\end{quote}

\subsubsection{Change of bilinear form}  \label{SSS:change}
Given two nonsingular forms $\beta, \gamma \in \bil_\text{assoc}(A;R)$, 
we obtain an isomorphism of left $A$-modules
$\bl_\beta^{-1} \circ \bl_{\gamma} \colon {}_AA \iso {}_AA$. Since this
isomorphism has the form $a \mapsto
au$ $(a \in A)$ for some unit $u \in A$, we see that
\[
\gamma(\,.\,,\,.\,) = \beta(\,.\,,\,.\,u) \ .
\]
If $\beta$ and $\gamma$ are both
symmetric then $\bl_\beta^{-1} \circ \bl_{\gamma}$
is an isomorphism of $(A,A)$-bimodules, and hence $u \in \cen(A)$, the center of $A$.


\section{Characters} \label{S:characters}

Throughout this section, 
$M$ will denote a left $A$-module that is assumed to be
finitely generated projective over $R$.
For $a \in A$, we let $a_M  \in \End_{R}(M)$ denote the endomorphism 
given by the action of $a$ on $M$: 
\[
a_M(m) = am \qquad (a \in A, m \in M)\ .
\]


\subsection{Trace and rank} \label{SS:trace}

The trace map 
\[
\tr \colon \End_{R}(M) \cong M \otimes_R M^\vee
\to R \ ;
\]
it is defined via evaluation of $M^\vee = \Hom_R(M,R)$ on $M$; see \cite[II.4.3]{nB70}.  
The image of the trace map complements the annihilator $\ann_RM = \{ r\in R \mid 
rm=0 \ \forall m \in M\}$:
\begin{equation} \label{E:ImTr}
\Im(\tr) \oplus \ann_RM = R \,;
\end{equation}
see \cite[Proposition I.1.9]{DeMI}.
The \emph{Hattori-Stallings rank} of $M$ is
defined by
\[
\rank_R M = \tr(1_M) \in R\ .
\]
If $M$ is free of rank $n$ over $R$ then $\rank_R M = n \cdot 1$. The following 
lemma is standard and easy.

\begin{lem} \label{L:trace}
The $R$-algebra $\End_{R}(M)$ is symmetric, with nonsingular associative symmetric 
$R$-bilinear form $\End_{R}(M)  \times \End_{R}(M) \to R$\,, $(x,y) \mapsto \tr(xy)$\,.
Identifying $\End_{R}(M)$ with $M \otimes_R M^\vee$ and writing $1_M = \sum m_i \otimes f_i$\,,
dual bases for this form are given by 
\begin{equation*} 
\{x_{i,j} = m_i \otimes f_j \}\,, \{ y_{i,j} = m_j \otimes f_i \} \ .
\end{equation*}
\end{lem}

\subsection{The character of $M$} \label{SS:character}

The \emph{character} of $M$ is the trace form $\chi_M \in A^\vee_{\trace}$ that is defined by
\[
\chi_M(a) =  \tr(a_M)  \qquad (a \in A)\ .
\]
If $e=e^2 \in A$ is an idempotent then 
$e_M = 1_{eM} \oplus 0_{(1-e)M}$, and so 
\begin{equation} \label{E:idempotentrank}
\chi_M(e) = \rank_R eM \ .
\end{equation}

Now assume that $A$ is Frobenius with associative nonsingular bilinear form $\beta$,
and let $\{ x_i \}_1^n,  \{ y_i \}_1^n \subseteq A$ be dual bases for $\beta$ as in 
\eqref{E:Frob3}.
Then the preimage of $\chi_M \in A^\vee_{\trace} \subseteq A^\vee$ under the isomorphism 
$\bl_\beta \colon {}_AA \iso {}_AA^\vee$ in \eqref{E:Frob2}
is the element 
\begin{equation} \label{E:characters}
z(M) = z_{\beta}(M):= \sum_i \chi_M(x_i)y_i \in A\ ;
\end{equation}
see \eqref{E:Frob4}. So 
\begin{equation} \label{E:characters1}
\chi_M(\,.\,) = \beta(\,.\,,z(M)) \ .
\end{equation}
In particular, $z(M)$ is independent of the choice of
dual bases $\{x_i \}, \{ y_i \}$. 

If $\beta$ is symmetric then $z(M) \in \cen(A)$ by \eqref{E:center}.

\subsection{The regular character} \label{SS:regular}
The left regular 
representation of $A$ is defined by $A \to \End_R(A)$, 
$a \mapsto (a_A \colon x \mapsto ax)$. Similarly, the right regular 
representation is given by $A \to \End_R(A)$, 
$a \mapsto ({}_Aa \colon x \mapsto xa)$.

If $A$ is Frobenius, with associative nonsingular bilinear form $\beta$
and  dual bases $\{ x_i \}_1^n,  \{ y_i \}_1^n \subseteq A$ for $\beta$,
then equations \eqref{E:Frob3} and \eqref{E:Frob3'} give the following expression
\[
a_A = \sum_i x_i \otimes \beta(a\,.\,,y_i) = \sum_i y_i \otimes \beta(x_i,a\,.\,) 
\in A \otimes_R A^\vee \cong \End_R(A)\ .
\]
Similarly, ${}_Aa = \sum_i x_i \otimes \beta(\,.\,a,y_i) = \sum_i y_i \otimes \beta(x_i,\,.\,a)$.
Taking traces, we obtain 
\begin{equation} \label{E:rightleft}
\tr(a_A) = \sum_i \beta(x_i,ay_i) = \sum_i \beta(x_ia,y_i) = \tr({}_Aa) \ .
\end{equation}
This trace is called the \emph{regular character} of $A$; it will be denoted by
$\chi_{\reg} \in A^\vee$. Since $\tr(a_A) = \sum_i \beta(ax_i,y_i)$ and 
$\tr({}_Aa) = \sum_i \beta(x_i,y_ia)$,
we have
\begin{equation} \label{E:characters2}
\chi_{\reg} = \beta(\,.\,,z) = \beta(z,\,.\,) \qquad \text{with} \quad z = z_\beta:= \sum_i x_iy_i \ .
\end{equation}
Thus, the element $z$ is associated to the regular character as in \eqref{E:characters1}.

\begin{example} \label{EX:regular}
We compute the regular character for the algebra $\End_R(M)$ using the trace form and the dual
bases $\{ x_{i,j}\}, \{ y_{i,j} \}$ from Lemma~\ref{L:trace}. The element $z$ in \eqref{E:characters2}
evaluates to
\[
z = \sum_{i,j} (m_i \otimes f_j)(m_j \otimes f_i) = \sum_{i,j} m_i f_j(m_j) \otimes f_i =
(\rank_R M) 1 \ .
\]
Therefore, the regular character of $\End_R(M)$ is equal to $\tr(\,.\,z) = (\rank_R M)\tr$\,.
\end{example}

\subsection{Central characters}  \label{SS:central}
Assume that $\End_A(M) \cong R$ as $R$-algebras. Then, for each $x \in \cen(A)$,
we have $x_M = \omega_M(x)1_M$ with $\omega_M(x) \in R$. This yields an 
$R$-algebra homomorphism
\[
\omega_M \colon \cen(A) \to R\ , 
\]
called the \emph{central character} of $M$. Since $\tr(x_M) = \omega_M(x)\tr(1_M)$,
we have 
\begin{equation} \label{E:centralcharacter}
\chi_M(x) = \omega_M(x) \rank_R M \qquad (x \in \cen(A)) \ .
\end{equation}

Now assume that $A$ is Frobenius with associative nonsingular bilinear form $\beta$,
and let $z(M) = z_\beta(M) \in A$ be as in \eqref{E:characters}. Then
\begin{equation} \label{E:centralcharacter2}
x z(M) = \omega_M(x) z(M) \qquad (x \in \cen(A)) \ ;
\end{equation}
this follows from the computation 
$\beta(a,x z(M)) = \beta(ax,z(M)) = \chi_M(ax) = \omega_M(x)\chi_M(a) \\
= \omega_M(x)\beta(a,z(M))
= \beta(a,\omega_M(x)z(M))$ for $a \in A$.

If $\beta$ is symmetric then $z(M) \in \cen(A)$ and we can
define the \emph{index}
\footnote{The index $[A:M]$ is also called the \emph{Schur element} of $M$; see, e.g., \cite[Lemma 4.6]{mB09}.}
of $M$ by
\begin{equation} \label{E:index}
[A:M]_\beta := \omega_M(z(M)) \in R \ .
\end{equation}

\subsection{Integrality} \label{SS:integrality}

Let $A$ be a Frobenius $R$-algebra, with associative nonsingular bilinear form $\beta$
and dual bases $\{ x_i \}_1^n,  \{ y_i \}_1^n \subseteq A$.
Assume that we are given a subring $S \subseteq R$. An $S$-subalgebra $B \subseteq A$
will be called a \emph{weak $S$-form of $(A,\beta)$} if the following conditions are satisfied:
\begin{enumerate}
\renewcommand{\labelenumi}{(\roman{enumi})}
\item 
$B$ is a finitely generated $S$-module, and
\item 
$\sum_i x_i \otimes y_i \in A \otimes_R A$ belongs to the image of 
$B \otimes_S B$ in $A \otimes_R A$.
\end{enumerate}
Recall from Section~\ref{SSS:dual} that the element 
$\sum_i x_i \otimes y_i$ only depends on $\beta$. Note also that (ii)
implies that $BR = A$. Indeed, for any $a \in A$, the map 
$\Id_A \otimes \beta(a,\,.\,) \colon A \otimes_R A \to A \otimes_R R = A$ 
sends $\sum_i x_i \otimes y_i$ to $a$ by \eqref{E:Frob3}, and it sends
the image of $B \otimes_S B$ in $A \otimes_R A$ to $BR$.

\begin{lem} \label{L:integrality}
Let $A$ be a symmetric  $R$-algebra with form $\beta$.
Assume that $\End_{A}(M) \cong R$, and  let $z(M) \in \cen(A)$ 
be as in \eqref{E:characters}.
If there exists a weak $S$-form of $(A,\beta)$ for some subring $S \subseteq R$,
then $\chi_M(z(M))$ and $[A:M]_\beta = \omega_M(z(M))$ 
are integral over $S$\,.
\end{lem}

\begin{proof}
All $b \in B$ are integral over $S$ by condition (i). Hence
the endomorphisms $b_M \in \End_{R}(M)$ are integral over $S$, and so are their
traces $\chi_M(b)$; see Prop.~17 in \cite[V.1.6]{nB64} for the latter. 
By (ii), it follows that $\chi_M(z(M)) = \sum_i \chi_M(x_i)\chi_M(y_i)$ is integral
over $S$. Moreover, the
subring $S' = S[\chi_M(B)] \subseteq R$ is finite over $S$ by (i).
Thus, all elements of $BS' \subseteq A$ are integral over $S$. In particular, 
this holds for $z(M)$, whence $\omega_M(z(M))$ 
is integral over $S$.
\end{proof}

\subsection{Idempotents} \label{SS:idempotents}

\begin{prop} \label{P:idempotents}
Let $(A,\beta)$ be a symmetric $R$-algebra. 
Let $e = e^2 \in A$ be an  idempotent and
assume that the $A$-module $M = Ae$ 
satisfies $\End_{A}(M) \cong R$ as $R$-algebras. Let
$\omega_M \colon \cen(A) \to R$ be the central character of $M$
and let $z(M) \in \cen(A)$ 
be as in \eqref{E:characters}. Then:
\begin{enumerate}
\item  
$[A:M]_\beta$ is invertible in $R$, with inverse $\beta(e,1)$.
\item 
$e(M):= [A:M]_\beta^{-1}z(M) \in \cen(A)$ is an idempotent satisfying
$\omega_M(e(M))=1$ and
\[
x e(M) = \omega_M(x) e(M) \qquad (x \in \cen(A))\ .
\]
\item 
Let 
$z = z_\beta = \sum_i x_iy_i \in \cen(A)$ be as in \eqref{E:characters2}. Then
\[
\omega_M(z) \cdot \rank_{R} M  = [A:M]_\beta  \cdot \rank_{R} e(M)A \ .
\]
\end{enumerate}
\end{prop}

\begin{proof}
Note that $eM =eAe \cong \End_A(M) \cong R$. 
By \eqref{E:idempotentrank}, it follows that $\chi_M(e)= \rank_{R}eM = 1$. 
Since $xe = \omega_M(x)e$ for $x \in \cen(A)$, we 
obtain
\[
1 = \chi_M(e) \underset{\eqref{E:characters1}}{=} \beta(e,z(M)) = 
\beta(z(M) e,1) = \omega_M(z(M))\beta(e,1) \ . 
\]
This proves (a).
In (b),  $\omega_M(e(M))=1$ is clear by definition of $e(M)$, and \eqref{E:centralcharacter2}
gives the identity
$x e(M) = \omega_M(x) e(M)$ for all $x \in \cen(A)$. 
Together, these facts imply that $e(M)$ is an idempotent.
Finally, the following computation
proves (c):
\[
\begin{split}
\omega_M(z) \rank_{R} M &\underset{\eqref{E:centralcharacter}}{=} \chi_M(z) 
\underset{\eqref{E:characters1}}{=} \beta(z,z(M)) \underset{\eqref{E:characters2}}{=} 
\chi_{\reg}(z(M)) \\
&= \omega_M(z(M)) \chi_{\reg}(e(M)) 
\underset{\eqref{E:idempotentrank}}{=} \omega_M(z(M)) \rank_{R} e(M)A \ .
\end{split} 
\]
\end{proof}

We now specialize the foregoing to separable algebras. For background, see 
DeMeyer and Ingraham  \cite{DeMI}.
We mention that, by a theorem of Endo and Watanabe \cite[Theorem 4.2]{sEyW67}, any
faithful separable $R$-algebra  is symmetric.

\begin{prop} \label{P:idempotents2}
Assume that the algebra $A$ is separable and that the $A$-module $M$ 
is cyclic and satisfies $\End_{A}(M) \cong R$.
Let $e(M) \in \cen(A)$ be the idempotent in Proposition~\ref{P:idempotents}(b).
Then:
\begin{enumerate}
\item
 $e(M)A \cong \End_R(M)$ and $\rank_{R} e(M)A =(\rank_{R} M)^2$\,.
\item
$\chi_{\reg}\, e(M) = (\rank_R M)\chi_M$\,, where  $\chi_{\reg}$ is the regular character of $A$\,.
\end{enumerate}
\end{prop}

\begin{proof}
We first note that $M$, being assumed projective over $R$, is in fact projective over $A$
by  \cite[Proposition II.2.3]{DeMI}. Since $M$ is cyclic, we have $M \cong Ae$ with $e = e^2 \in A$;
so Proposition~\ref{P:idempotents} applies.

(a) 
It suffices to show that $e(M)A \cong \End_R(M)$, because the rank of $\End_R(M) \cong M \otimes_R M^\vee$
equals $(\rank_{R} M)^2$. Since $M$ is finitely generated projective and faithful over $R$,
the  $R$-algebra $\End_R(M)$ is Azumaya; see \cite[Proposition II.4.1]{DeMI}.  
The Double Centralizer Theorem \cite[Proposition II.1.11 and Theorem II.4.3]{DeMI} and our 
hypothesis $\End_{A}(M) \cong R$ together imply
that the map $A \to \End_R(M)$, $a \mapsto a_M$, is surjective. Letting $I = \ann_AM$ denote the
kernel of this map, we further know by \cite[Corollary II.3.7 and Theorem II.3.8]{DeMI} that
$I = (I \cap \cen(A))A$.
Finally, Proposition~\ref{P:idempotents} tells us that $I \cap \cen(A)$ is generated by $1-e(M)$, 
which proves (a).

(b)
In view of the isomorphism $e(M)A \cong \End_R(M)$, $e(M)a \mapsto a_M$ from part (a) and 
Example~\ref{EX:regular},  we have $\chi_{\reg}(e(M)a) = (\rank_RM)\tr(a_M) = (\rank_RM)\chi_M(a)$\,.
\end{proof}


\section{Separability} \label{S:separability}

The $R$-algebra $A$ is assumed to be Frobenius throughout this section.
We
fix a nonsingular associative $R$-bilinear form $\beta \colon A \times A \to R$ and dual bases
$\{ x_i \}_1^n,  \{ y_i \}_1^n \subseteq A$ for $\beta$.


\subsection{The Casimir operator} \label{SS:Casimir}
Define a map, called the \emph{Casimir operator} 
\footnote{We follow the terminology of Higman's original article \cite{dH55}; the map $c$ is called the
\emph{Gasch{\"u}tz-Ikeda operator} in \cite{cCiR62}.}
of $(A,\beta)$, by
\begin{equation} \label{E:casimir}
c = c_\beta \colon A \to \cen(A) \ , \quad a \mapsto \sum_i y_i a x_i \ .
\end{equation}
In order to check that $c(a) \in \cen(A)$ we calculate, for $a,b \in A$,
\[
b c(a) \underset{\eqref{E:Frob3'}}{=} \sum_{i,j} \beta(x_j,by_i)y_jax_i 
= \sum_{i,j} y_ja\beta(x_jb,y_i)x_i \underset{\eqref{E:Frob3}}{=} c(a)b \ .
\]
The map $c$ is independent of the choice of
dual bases $\{ x_i \},  \{ y_i \}$, because 
$\sum_i x_i \otimes y_i \in A \otimes_R A$
only depends on $\beta$; see Section~\ref{SSS:dual}.

In case $\beta$ is symmetric,  $\{ y_i \},  \{ x_i \}$ are also dual
bases for $\beta$, and hence 
\[
c(a) = \sum_i x_i a y_i \ .
\]
In particular, the element $z = z_\beta$ in \eqref{E:characters2} arises as $z_\beta = c(1)$ if $\beta$
is symmetric. We will refer to
the element $z = z_\beta$ as the 
\emph{Casimir element}
\footnote{In \cite{mB09}, the Casimir element $z$ is referred to as the \emph{central Casimir element}, while
$\sum_i x_i \otimes y_i \in A \otimes_R A$ is called the Casimir element. The latter element is referred to
as the \emph{Frobenius element} in \cite{lK99}.} 
of the symmetric algebra $(A,\beta)$; it depends on
$\beta$ only up to a central unit (see \ref{SS:CasimirIdeal} below).

\subsection{The Casimir ideal}   \label{SS:CasimirIdeal}

Since $c$ is $\cen(A)$-linear, the image $c(A)$ of the map $c = c_\beta$ 
in \eqref{E:casimir} is an ideal of 
$\cen(A)$. This ideal will be called the \emph{Casimir ideal}
\footnote{The Casimir ideal $c(A)$ of a symmetric algebra $A$ coincides with the \emph{projective center} of 
$A$ in the terminology of \cite[Prop.~3.13]{mB09}. 
It is also called the \emph{Higman ideal} of $A$; see \cite{HHKM}.} 
of $A$; it does not depend on the choice of the 
bilinear form $\beta$. Indeed, recall from Section \ref{SSS:change} 
that any two nonsingular forms $\beta, \gamma \in \bil_\text{assoc}(A;R)$ 
are related by $\gamma(\,.\,,a) = \beta(\,.\,,au)$ for
some unit $u \in A$. Hence, if $\{ x_i \},  \{ y_i \} \subseteq A$ 
are dual bases for $\beta$ then $\{ x_i \},  \{ y_i u^{-1} \} \subseteq A$ 
are dual bases for $\gamma$. Therefore, 
\begin{equation*}
c_{\gamma}(a) = c_\beta(u^{-1}a) \qquad  (a \in A)\ .
\end{equation*}
If $\beta$ and $\gamma$ are both
symmetric then $u \in \cen(A)$ and so 
$c_{\gamma}(a) = u^{-1}c_\beta(a)$.

\subsection{The separability locus} \label{SS:separability}
 
For a given Frobenius algebra $A$, we will now determine the 
set of all primes $\fp \in \Spec R$ such that 
the $Q(R/\fp)$-algebra
$A \otimes_{R} Q(R/\fp)$ is separable or, equivalently, the $R_\fp$-algebra
$A \otimes_RR_\fp$ is separable \cite[Theorem II.7.1]{DeMI}.
The collection of these primes is called the separability locus of $A$.

\begin{prop} \label{P:separability}
The separability locus of a Frobenius $R$-algebra $A$ is
\[
\Spec R \setminus V(c(A)\cap R) = \{ \fp \in \Spec R \mid 
\fp \nsupseteq c(A) \cap R\}\ .
\]
\end{prop}
 
\begin{proof}
The case of a base field $R$ is covered by
Higman's Theorem which states that a Frobenius algebra $A$ over a field $R$ is separable if and only if
$c(A) = \cen(A)$ or, equivalently, $1 \in c(A)$; see \cite[Theorem 1]{dH55} or \cite[71.6]{cCiR62}.

Now let $R$ be arbitrary and let $\fp \in \Spec R$ be given. Put 
$F = Q(R/\fp)$ and $A_F = A \otimes_R F$.
We know that $A_F$ is Frobenius, with form $\bar{\beta} = \beta \otimes_R\Id_F$ and 
corresponding dual bases $\{ \bar{x}_i \}$, $\{ \bar{y}_i \}$, where 
$\bar{\phantom{x}} \colon A \to A_F$, $\bar{x} = x \otimes 1$, denotes the canonical map.
By Higman's Theorem, we know that 
$A_F$ is separable if and only if $1 \in c(A_F) = \bar{c(A)}F$. Thus:
\begin{quote}
The $F$-algebra $A_F$ is separable if and only if $\left( A\fp + c(A) \right) \cap R \supsetneqq \fp$.
\end{quote}
If  $\fp \nsupseteq c(A) \cap R$ then clearly $\left( A\fp + c(A) \right) \cap R \supsetneqq \fp$, and hence
$A_F$ is separable. 
Conversely, assume that $c(A) \cap R \subseteq \fp$. Since $A$ is integral over its center $\cen(A)$, we have
$c(A)A \cap \cen(A) \subseteq \sqrt{c(A)}$, the radical of the ideal $c(A)$; see Lemma 1 in \cite[V.1.1]{nB64}.
Therefore, $c(A)A \cap R \subseteq \sqrt{c(A)} \cap R \subseteq \fp$. By 
Going Up  \cite[13.8.14]{jMcCjR87}, there exists a prime ideal 
$P$ of $A$ with $c(A)A \subseteq P$ and $P \cap R = \fp$. But then 
$ \left( A\fp + c(A) \right) \cap R \subseteq P \cap R = \fp$, and hence
$A_F$ is not separable. This proves the proposition.
\end{proof}

\subsection{Norms} \label{SS:norms}

Assume that the algebra $A$ is free of rank $n$ over $R$.
Then the \emph{norm} of an element $a \in A$ is defined by 
\[
N(a) = \det a_A  \in R \ ,
\]
where $(a_A \colon x \mapsto ax) \in \End_R(A) = \Mat_n(R)$ is the left
regular representation of $A$ as 
in Section \ref{SS:regular}.
The norm map $N \colon A \to R$ satisfies
$N(ab) = N(a)N(b)$ and $N(r) = r^n$
for $a,b \in A$ and $r \in R$. 
Up to sign, $N(a)$ is the constant term of  
the characteristic polynomial of $a_A$.
Since $a$ satisfies this polynomial by the Cayley-Hamilton Theorem, we see
that $a$ divides $N(a)$ in $R[a] \subseteq A$. Putting
\[
N(c(A)) = \sum_{a\in c(A) }RN(a) \ ,
\]
we obtain
$N(c(A)) \subseteq c(A) \cap R$. Moreover, since $N(r) = r^n$ for $r \in R$, 
we further conclude that, for any $\fp \in \Spec R$, we have
\[
\fp \supseteq N(c(A)) \iff \fp \supseteq c(A) \cap R\ .
\]


\section{Additional structure: augmentations, involutions, positivity} \label{S:additional}

\subsection{Augmentations and integrals}  \label{SS:augmentation}

Let $(A,\beta)$ be a Frobenius algebra and suppose that $A$ has an augmentation, that is,
an algebra homomorphism 
\[
\e \colon A \to R \ .
\]
Put $\Lambda_\beta = \br_\beta^{-1}(\e) \in A$, where $\br_\beta$ is as in Section~\ref{SS:nonsingular}; so
$\beta(\Lambda_\beta,\,.\,) = \e$. From \eqref{E:Frob3'}, we obtain the following expression in terms of dual 
bases $\{ x_i\}, \{y_i\}$ for $\beta$:
\begin{equation} \label{E:augmentation1}
\Lambda_\beta = \sum_i \e(y_i)x_i \ .
\end{equation}
The computation $\beta(\Lambda_\beta a, \,.\,) = \beta(\Lambda_\beta, a \,.\,) = \e(a)\e 
= \e(a) \beta(\Lambda_\beta,\,.\,)$ for all $a \in A$ shows that $\Lambda_\beta a = \e(a)\Lambda_\beta$.
Conversely, if $t \in A$ satisfies $t a = \e(a)t$ for all $a \in A$ then $\beta(t, a) =
\beta(t a,1) =  \e(a) \beta(t, 1)$, whence $t = \beta(t, 1)\Lambda_\beta$.
We put
\[
\Hint_A^r = \{ t \in A \mid t a = \e(a)t \text{ for all $a \in A$} \}
\]
and call the elements of $\Hint_A^r$ \emph{right integrals} in $A$. The foregoing shows that
$\Hint_A^r = R \Lambda_\beta$. Moreover, $r\br_\beta^{-1}(\e)=0$ implies $r\e = 0$ and hence $r=0$.
Thus:
\[
\Hint_A^r = R\Lambda_\beta \cong R \ .
\]

Similarly, one can define the $R$-module $\Hint_A^l$ of \emph{left integrals} in $A$ and show that
\[
\Hint_A^l = R \Lambda_\beta' \cong R \qquad \text{with } \Lambda_\beta' = \sum_i \e(x_i)y_i = \bl_\beta^{-1}(\e) \  .
\]
Define the ideal $\edim A$ of $R$ by
\begin{equation} \label{E:augmentation2}
\edim A := \e(\Hint_A^r) = \e(\Hint_A^l) = \e(c(A)) = (\e(z)) \ ,
\end{equation}
where $c(A)$ is the Casimir ideal and $z = z_{\beta} \in \cen(A)$ is as in \eqref{E:characters2}. 
Note that always $c(A) \cap R \subseteq \e(c(A))$; so
\begin{equation} \label{E:augmentation3}
c(A) \cap R \subseteq  \edim A \ .
\end{equation}

If $\beta$ is symmetric then $\br_\beta = \bl_\beta$ and hence $\Lambda_\beta = \Lambda_\beta'$ and
$\Hint_A^r = \Hint_A^l =:\Hint_A$. 
For further information on the material in this section, see \cite[6.1]{lK99}.

\subsection{Involutions} \label{SS:involutions}

Let $A$ be a symmetric algebra with symmetric associative
bilinear form $\beta \colon A \times A \to R$.  
Suppose further that $A$ has an involution $^*$, that is, an $R$-linear endomorphism of $A$ 
satisfying $(xy)^* = y^*x^*$ and $x^{**} = x$ for all $x,y \in A$.
If $A$ is 
$R$-free with basis $\{ x_i \}_1^n$ satisfying
\begin{equation} \label{E:involutions1}
\beta(x_i,x_j^*) = \delta_{i,j}\ ,
\end{equation}
then we will call $A$ a \emph{symmetric $*$-algebra}. The Casimir operator
$c = c_\beta \colon A \to \cen(A)$ takes the form
\begin{equation*} 
c(a) = \sum_i x_i^* a x_i = \sum_i x_i a x_i^* \ ,
\end{equation*}
and the Casimir element $z = z_\beta = c(1)$ is
\begin{equation} \label{E:involutions2}
z = \sum_i x_i^* x_i = \sum_i x_i x_i^*\ .
\end{equation}

\begin{lem} \label{L:involutions}
Let $(A,\beta,*)$ be a symmetric $*$-algebra.
Then:
\begin{enumerate}
\item  $\beta$ is $*$-invariant: $\beta(x,y) = \beta(x^*,y^*)$ for all $x,y \in A$.
\item The Casimir operator $c$ is $*$-equivariant: $c(a)^* = c(a^*)$ for all $a \in A$.
In particular, $z^* = z$.
\item If $a=a^* \in \cen(A)$ then the matrix of $(a_A \colon x \mapsto ax) \in \End_R(A)$ 
with respect to the $R$-basis $\{ x_i \}$ of $A$ is symmetric.
\end{enumerate}
\end{lem}
 
\begin{proof} Part (a) follows from $\beta(x_i,x_j^*) = \delta_{i,j} = \beta(x_j,x_i^*) = \beta(x_i^*,x_j)$,
and (b) follows from $c_\beta(a)^* = \sum (x_i^* a x_i)^* = \sum_i x_i^* a^* x_i = c_\beta(a^*)$.

(c) Let $(a_{i,j}) \in \Mat_n(R)$ be the matrix of $a_A$; so $a x_j = \sum_i a_{i,j} x_i$.
We compute using associativity, symmetry 
and $*$-invariance of $\beta$:
\[
a_{i,j} = \beta(x_i^*,ax_j) =  \beta(ax_i^*,x_j) =  \beta(x_j,ax_i^*) = \beta(x_j^*, ax_i) = a_{j,i} \ .
\]
\end{proof}

\subsection{Positivity} \label{SS:positivity}

Let $(A,\beta,*)$ be a symmetric $*$-algebra with $R$-basis $\{ x_i \}_1^n$ satisfying
\eqref{E:involutions1}. Assume that
$R \subseteq \RR$ and put $R_+ = \{ r \in R \mid r \ge 0 \}$. 
If
\begin{equation} \label{E:positivity}
A_+ := \bigoplus_i R_+ x_i \text{ is closed under multiplication and stable under $*$ }
\end{equation}
then we will then say that $A$ has a \emph{positive structure} and call $A_+$ the \emph{positive 
cone} of $A$. 

We now consider the endomorphism $(z_A \colon x \mapsto zx) \in \End_R(A)$
for the Casimir element $z = z_\beta$ in \eqref{E:involutions2}. 
By Lemma~\ref{L:involutions}, we know that the matrix
of $z_A$ with respect to the basis $\{ x_i \}$ is symmetric.
The following proposition gives further information.

\begin{prop} \label{P:positivity}
Let $(A,\beta,*)$ be a symmetric $*$-algebra over the ring $R \subseteq \RR$,
and let $z = z_\beta$ be the Casimir element.
\begin{enumerate}
\item 
The matrix of $z_A$ with respect to the basis $\{ x_i \}$  is symmetric and
positive definite. In particular, all eigenvalues of $z_A$ are positive real numbers that 
are integral over $R$.
\item 
If $A$ has a positive structure and an augmentation $\e \colon A \to R$ satisfying
$\e(a) > 0$ for all $0 \neq a \in A_+$.
Then the largest eigenvalue of $z_A$ is $\e(z)$.
\end{enumerate}

\end{prop}

\begin{proof}
(a)
Let $Z = (z_{i,j})$ be the matrix of $z_A$; so $z_{i,j} = \beta(x_i^*,zx_j)$.
Extending $*$ and $\beta$ to $A_\RR = A \otimes_R\RR$ by linearity, one computes
for $x = \sum_i \xi_i x_i \in A_\RR$:
\[
\sum_l \beta((x_lx)^*, x_lx) = \beta(x^*,zx) = (\xi_1,\dots\xi_n) Z (\xi_1,\dots\xi_n)^{\text{tr}} \ .
\]
The sum on the left is positive if $x\neq 0$, because $\beta(y^*,y) = \sum \eta_i^2$ for 
$y = \sum_i \eta_i x_i \in A_\RR$. This shows that $Z$ is positive definite. The assertion about the
eigenvalues of $Z$ is a standard fact about positive definite symmetric matrices over the reals.

(b)
By hypothesis on $A_+$, the matrix of $a_A$ 
with respect to the basis $\{ x_i \}$ has non-negative entries for any $a \in A_+$.
Moreover, the Casimir element $z= \sum_i x_i^*x_i$ belongs to $A_+$, and so
$z_A$ is non-negative.
Now let $\Lambda  = \sum_i \e(x_i^*)x_i \in \Hint_A$ be the integral of $A$ that is associated to 
the augmentation $\e$; see \eqref{E:augmentation1}. Then 
$\Lambda$ is an eigenvector for $z_A$ with eigenvalue $\e(z)$. Since all $\e(x_i^*) > 0$, it
follows that $\e(z)$ is in fact the largest (Frobenius-Perron) eigenvalue of $z_A$ is $\dim_\k H$; 
see \cite[Chapter XIII, Remark 3 on p.~63/4]{fG2}.
\end{proof}

\begin{cor} \label{C:positivity}
If $A$ is a symmetric $*$-algebra over the ring $R \subseteq \RR$, then the Casimir element
$z$ is a regular element of $A$. Furthermore, $A \otimes_R Q(R)$ is separable.
\end{cor}

\begin{proof}
Regularity of $z$ is clear from Proposition~\ref{P:positivity}(a). Since $z$ is integral over $R$, it follows that
$z\cen(A) \cap R \neq 0$. Therefore, $c(A) \cap R \neq 0$ and Proposition~\ref{P:separability} 
gives that $A \otimes_R Q(R)$ is separable.
\end{proof}


\part{HOPF ALGEBRAS} \label{2}

Throughout this part, $H$ will denote a finitely generated projective
Hopf algebra over the commutative ring $R$ (which will be assumed to be a field in
Section~\ref{S:Semisimple}), with
unit $u$, multiplication $m$, counit $\e$, comultiplication $\Delta$, and antipode $\sS$.
We will use the Sweedler notation $\Delta h = \sum h_1 \otimes h_2$.

In addition to the bimodule action of $H$ on $H^\vee$
in \eqref{E:VeeAction}, we now also have an analogous bimodule action of
the dual algebra $H^\vee$ on $H = H^{\vee\vee}$. In order to avoid confusion, it is
customary to indicate the target of the various actions by $\rightharpoonup$ or
$\leftharpoonup$\,:
\begin{equation} \label{E:VeeAction'}
\begin{aligned} 
\langle  a \rightharpoonup f \leftharpoonup b , c \rangle &= \langle f, bca \rangle \qquad
&  &(a,b,c \in H, f \in H^\vee)\ \,,\\
\langle  e , f \rightharpoonup a \leftharpoonup  g \rangle &= \langle gef, a \rangle
&  &(e,f,g \in H^\vee, a \in H)\,.
\end{aligned} 
\end{equation}
Here and for the remainder of this article, $\langle  \,.\, , \,.\, \rangle \colon H^\vee \times H \to R$
denotes the evaluation pairing.


\section{Frobenius Hopf algebras over commutative rings} \label{S:FrobeniusHopf}


\subsection{} \label{SS:Hopf2}

The following result is
due to Larson-Sweedler \cite{rLmS69}, Pareigis \cite{bP71}, and
Oberst-Schneider  \cite{uOhS73}. 

\begin{thm} \label{T:Hopf}
\begin{enumerate}
\item 
The antipode $\sS$ is bijective. Consequently, $\Hint^l_H = \sS(\Hint^r_H )$.
\item
$H$ is a Frobenius $R$-algebra if and only if 
$\Hint^r_H \cong R$. This always holds if $\Pic R = 1$.
Furthermore, if $H$ is Frobenius then so is the dual algebra $H^\vee$.
\item
Assume that $H$ is Frobenius. Then $H$ is symmetric if and only if 
\begin{enumerate}
\item 
$H$ is unimodular (i.e., $\Hint^l_H = \Hint^r_H$), and
\item
$\sS^2$ is an inner automorphism of $H$.
\end{enumerate}
\end{enumerate}
\end{thm}

\begin{proof}
Part (a) is  \cite[Proposition 4]{bP71} and (c) is \cite[3.3(2)]{uOhS73}.
For necessity of the condition $\int^r_H \cong R$ in (b), in the 
more general context of 
augmented Frobenius algebras, see Section~\ref{SS:augmentation}.
Conversely, if $\int^r_H \cong R$ holds then \cite[Theorem 1]{bP71} asserts that the dual
algebra $H^\vee$ is Frobenius. This forces  $\int_{H^\vee}^r$ to be free of rank $1$ over $R$, and hence
$H$ is Frobenius by \cite[Theorem 1]{bP71}.
The statement about $\Pic R = 1$ is a consequence of the fact that the $R$-module $\int^r_H$ is
invertible (i.e., locally free of rank $1$) for any  finitely generated projective Hopf $R$-algebra $H$; see 
\cite[Proposition 3]{bP71}. 
\end{proof}

\subsection{} \label{SS:Hopf3}

We spell out some of the data associated with a Frobenius Hopf algebra $H$ referring
the reader to the aforementioned references \cite{rLmS69},  \cite{bP71},  \cite{uOhS73}
for complete details. 

Fix a generator $\Lambda \in \int_H^r$\,.
There is a unique
$\lambda \in \Hint_{H^\vee}^l$ satisfying $\lambda \leftharpoonup \Lambda = \e$ or, equivalently,
$\langle \lambda, \Lambda \rangle = 1$.  Note that this equation implies that $\Hint_{H^\vee}^l = R \, \lambda$,
because $\Hint_{H^\vee}^l$ is an invertible $R$-module. A
nonsingular associative bilinear form $\beta = \beta_\lambda$ for $H$ is given by
\begin{equation} \label{E:Hopf1}
\beta(a,b) = \langle \lambda, ab \rangle  \qquad (a,b \in H) \ .
\end{equation}
Dual bases $\{ x_i \}$, $\{ y_i \}$ for $\beta$ are given by 
$\sum x_i \otimes y_i = \sum \Lambda_2 \otimes \sS(\Lambda_1)$:
\begin{equation} \label{E:Hopf2}
a = \sum \langle \lambda,a \sS(\Lambda_1)\rangle \Lambda_2 = 
\sum \langle \lambda , \Lambda_2a \rangle \sS(\Lambda_1) \qquad (a \in H)\,.
\end{equation}
By \cite[p.~83]{rLmS69}, the form $\beta$ is \emph{orthogonal} for the right action $\leftharpoonup$ of
$H^\vee$ on $H$\,:
\begin{equation} \label{E:Hopf2'}
\beta(a,b \leftharpoonup f) = \beta(a \leftharpoonup \sS^\vee(f), b) \qquad (a,b\in H, f \in H^\vee)\ ,
\end{equation}
where $\sS^\vee = \,.\,\circ\sS$ is the antipode of $H^\vee$.

\subsection{} \label{SS:SepLocus}

By \eqref{E:Hopf2} the Casimir operator has the form
\[
c  = c_\Lambda \colon H \to \cen(H) \ , \quad  a \mapsto \sum \sS(\Lambda_1)a\Lambda_2 \ .
\]
In particular, $c(1) = \langle \e,\Lambda \rangle \in R$. Therefore, equality holds in \eqref{E:augmentation3}:
\begin{equation} \label{E:Hopf2''}
\edim H = \langle \e,\Hint_H^r \rangle = \langle \e,\Hint_H^l \rangle = c(H) \cap R \ .
\end{equation}
Proposition~\ref{P:separability} 
now gives the following classical result of Larson and Sweedler \cite{rLmS69}.

\begin{cor} \label{C:Hopfseparability}
The separability locus of a Frobenius Hopf algebra
$H$ over $R$ is
\[
\Spec R \setminus V(\edim H)\,.
\]
In particular, $H$ is separable if and only if $\langle \e,\Lambda \rangle = 1$ for some 
right or left integral $\Lambda \in H$.
\end{cor}

The equality $\langle \e,\Lambda \rangle = 1$ implies that $\Lambda$ is an idempotent 
two-sided integral such that
$\Hint^r_H = \Hint^l_H = R\Lambda$, because $\Hint^r_H$ and $\Hint^l_H$
are invertible $R$-modules.

\subsection{} \label{SS:Hopf4}

Let $H$ be a Frobenius Hopf algebra over $R$, and let $\Lambda \in \int_H^r$ 
and $\lambda \in \Hint_{H^\vee}^l$ be as in \ref{SS:Hopf3}; so 
$\langle \lambda, \Lambda \rangle = 1$.
The isomorphisms $\br_\beta$ and $\bl_\beta$ in \eqref{E:Frob1} and
\eqref{E:Frob2} for the the bilinear form $\beta = \beta_\lambda$ in \eqref{E:Hopf1}
will now be denoted by $\br_\lambda$ and $\bl_{\lambda}$, respectively:
\begin{equation} \label{E:Hopf3}
\begin{aligned} 
\br_{\lambda} &\colon H_H \iso H^\vee_H\ , &  a &\mapsto ( \beta(a,\,.\,) = 
\lambda \leftharpoonup a ) \ , \\
\bl_{\lambda} &\colon {}_HH \iso {}_HH^\vee\,,   & a &\mapsto ( \beta(\,.\,,a) = 
a \rightharpoonup \lambda ) \ .
\end{aligned}
\end{equation}
Equation \eqref{E:Hopf2'} states that
\begin{equation} \label{E:Hopf3'}
\br_{\lambda}(a \leftharpoonup \sS^\vee(f))  = f \br_{\lambda}(a) 
\qquad\text{and}\qquad
\bl_{\lambda}(a \leftharpoonup f)  = \sS^\vee(f) \bl_{\lambda}(a) \ .
\end{equation}
for $a \in H$, $f \in H^\vee$. Since $\br_{\lambda}(\Lambda) = \e$ is the identity of ${H^\vee}$, we obtain the
following expression for the inverse of $\br_\lambda$:
\begin{equation} \label{E:Hopf3''}
\br_{\lambda}^{-1} \colon H^\vee_H \iso H_H\ , \quad
f \mapsto ( \Lambda \leftharpoonup \sS^\vee(f) ) \ .
\end{equation}

\subsection{} \label{SS:separable}

In contrast with the Frobenius property, symmetry does not generally pass from $H$ to $H^\vee$; see
\cite[2.5]{mL97a}. We will call $H$ \emph{bi-symmetric} if both $H$ and $H^\vee$ are symmetric.

\begin{lem} \label{L:sep}
Assume that $H$ is Frobenius and fix
integrals $\Lambda \in \int_H^r$ 
and $\lambda \in \Hint_{H^\vee}^l$ such that 
$\langle \lambda, \Lambda \rangle = 1$, as in Sections~\ref{SS:Hopf3}, \ref{SS:Hopf4}. Then:
\begin{enumerate}
\item
If $H$ is involutory (i.e., $\sS^2 = 1$) then the regular character $\chi_{\reg}$  of $H$ is given by
$\chi_{\reg} = \langle \e,\Lambda \rangle \lambda$\,.
\item 
$H$ is symmetric and involutory if and only if $H$ is unimodular
and all left and right integrals in $H^\vee$ belong to $H^\vee_{\trace}$.
\item
Let $H$ be separable and involutory. Then $H$ is bi-symmetric. Furthermore, 
\[
\Hint_{H^\vee} = R\, \chi_{\reg}
\qquad \text{and} \qquad
\operatorname{Dim}_{u^\vee} H^\vee = \left( \rank_R H \right)\ ,
\]
where $u^\vee = \langle \,.\,,1 \rangle$ the counit of $H^\vee$.
\end{enumerate} 
\end{lem}

\begin{proof}
(a)
Equations \eqref{E:characters2}, \eqref{E:Hopf1} and \eqref{E:Hopf2},
with $z = \sum_i x_iy_i = \sum \Lambda_2\sS(\Lambda_1) = \langle \e,\Lambda \rangle$
(using $\sS^2 = 1$), give
$\chi_{\reg} = \langle \lambda, \,.\,z \rangle = \langle \e,\Lambda \rangle \lambda$\,.

(b)
First assume that $H$ is symmetric and involutory. Then $H$ is unimodular by
Theorem~\ref{T:Hopf}(c). By \cite[3.3(1)]{uOhS73} we further know that the
Nakayama automorphism of $H$ is equal to $\sS^2$, and hence it 
is the identity. 
Thus, $\lambda \leftharpoonup a = a \rightharpoonup \lambda$ for all $a \in H$,
which says that $\lambda$ is a trace form. Hence, $\Hint_{H^\vee}^l \subseteq H^\vee_{\trace}$\,. 
Since $H^\vee_{\trace}$ is stable under
the antipode $\sS^\vee$ of $H^\vee$, it also contains $\Hint_{H^\vee}^r$\,. The converse follows
by retracing these steps.

(c)
Now let $H$ be separable and involutory. By Corollary~\ref{C:Hopfseparability} and the
subsequent remark, $H$ is unimodular and we may choose
$\Lambda \in  \Hint_H$ such that $\langle \e,\Lambda \rangle = 1$. Part (a) gives 
$\Hint_{H^\vee}^l = R\, \chi_{\reg}$\,. The computation
\[
\langle \sS^\vee(\chi_{\reg}), a \rangle = \tr(\sS(a)_A) = \tr(\sS\circ{}_Aa\circ\sS^{-1}) = \tr({}_Aa) 
\underset{\eqref{E:rightleft}}{=} \langle \chi_{\reg}, a \rangle
\]
for $a \in A$ shows that $\sS^\vee(\chi_{\reg}) = \chi_{\reg}$\,. Therefore, we also have 
$\Hint_{H^\vee}^r = R\, \chi_{\reg}$\,.
In view of Theorem~\ref{T:Hopf}(c), this shows that $H$ is bi-symmetric.
Finally, since $\langle \chi_{\reg},1 \rangle = \rank_R H$\,,
equation \eqref{E:Hopf2''} yields
$\operatorname{Dim}_{u^\vee} H^\vee = \left( \rank_R H \right)$.
\end{proof}

As was observed in the proof of (b), the maps  $\br_\lambda$ and $\bl_\lambda$ in \eqref{E:Hopf3}
coincide if and only if $\lambda$ is a trace form. 
In this case, we will denote the $(H,H)$-bimodule isomorphism  $\br_\lambda = \bl_\lambda$ by
$\bb_\lambda$:
\begin{equation} \label{E:Hopf4}
\bb_{\lambda} \colon {}_HH_H \iso {}_H{H^\vee\!}_H\ ,   \quad a \mapsto 
( \lambda \leftharpoonup a  = a \rightharpoonup \lambda ) \ .
\end{equation}
We also remark that the formula $\operatorname{Dim}_{u^\vee} H^\vee = \left( \rank_R H \right)$
in (c) is a special case of the following formula which holds for any involutory $H$\,; see \cite[3.6]{uOhS73}:
\begin{equation} \label{E:DimProd}
\edim H \cdot \operatorname{Dim}_{u^\vee} H^\vee = \left( \rank_R H \right)\ .
\end{equation}

\subsection{} \label{SS:Hopf5}

We let 
\[
\Cocom(H) = \{ a \in H \mid {\textstyle \sum} a_1 \otimes a_2 
= {\textstyle \sum} a_2 \otimes a_1\}
\]
denote the $R$-subalgebra of $H$ consisting of all cocommutative 
elements. Thus, $\Cocom(H^\vee) = H^{\vee}_{\trace}$\,.
Recall from Lemma~\ref{L:sep}(b) that all integrals in $H^\vee$ are
cocommutative if $H$ is symmetric and involutory.

\begin{lem} \label{L:Hopf}
Let $H$ be bi-symmetric and involutory. Fix
a generator $\lambda \in \Hint_{H^\vee}$.
Then the $(H,H)$-bimodule isomorphism 
$\bb_\lambda$ in \eqref{E:Hopf4}
restricts to an isomorphisms
\[
\cen(H) \iso H^{\vee}_{\trace} \qquad \text{and} \qquad \Cocom(H) \iso \cen(H^\vee)\ .
\]
\end{lem}

\begin{proof}
The isomorphism $\cen(H) \iso H^{\vee}_{\trace} = \Cocom(H^\vee)$ is \eqref{E:center}.
By the same token, fixing 
a (necessarily cocommutative) generator $\Lambda \in \int_H$ 
such that $\langle \lambda, \Lambda \rangle = 1$, we obtain that
$\bb_\Lambda$ is an $(H^\vee,H^\vee)$-bimodule isomorphism
$H^\vee \iso H^{\vee\vee} = H$ that restricts to an isomorphism $\cen(H^\vee) \iso \Cocom(H)$.
By equation \eqref{E:Hopf3''} we have 
\begin{equation} \label{E:bL}
\bb_\Lambda(\sS^\vee(f)) = \bb_\lambda^{-1}(f)
\end{equation}
for $f \in H^\vee$. Since $\cen(H^\vee)$ is stable under the antipode $\sS^\vee$ of $H^\vee$,
we conclude that $\bb_\lambda^{-1}$ restricts to an isomorphism $\cen(H^\vee) \iso \Cocom(H)$,
and hence  $\bb_\lambda$ restricts to an isomorphism $\Cocom(H) \iso \cen(H^\vee)$.
\end{proof}

\subsection{} \label{SS:Hopf6}

Turning to modules now, we review some standard constructions and facts. For
any two left $H$-modules $M$ and $N$, the tensor product $M \otimes_R N$ becomes an
$H$-module via $\Delta$\,, and $\Hom_R(M,N)$ carries the following $H$-module structure:
\[
(a\phi)(m) = \sum a_1\phi(\sS(a_2)m)
\]
for $a\in H, m\in M, \phi \in \Hom_{R}(M,N)$\,. In particular, viewing $R$ as $H$-module
via $\e$, the $H$-action on the dual $M^\vee = \Hom_R(M,R)$ takes the following form:
\[
\langle af, m \rangle  = \langle f,\sS(a)m \rangle
\]
for $a\in H, m\in M, f \in M^\vee$\,. The $H$-invariants in
$\Hom_R(M,N)$ are exactly the $H$-module maps:
\[
\Hom_R(M,N)^H = \{ \phi \in \Hom_{R}(M,N) \mid a\phi = \langle \e,a \rangle \phi \ \forall
a\in A \} = \Hom_H(M,N)\ ;
\]
see, e.g., \cite[Lemma 1]{yZ94}. Moreover, it is easily checked that the canonical map
\[
N \otimes_R M^\vee \to \Hom_R(M,N)\ , \qquad n \otimes f \mapsto (m \mapsto \langle f,m \rangle n)
\]
is a homomorphism of $H$-modules. This map is an isomorphism if $M$ or $N$ is
finitely generated projective over $R$; see \cite[II.4.2]{nB70}. 

Finally, we consider the
trace map $\tr \colon \End_{R}(M) \cong M \otimes_R M^\vee \to R$ of Section~\ref{SS:trace}.

\begin{lem} \label{L:Hopf6}
Let $H$ be involutory and let $M$ be a left $H$-module that is finitely generated generated projective over $R$.
Then the trace map $\tr$  is an $H$-module map.
\end{lem}

\begin{proof}
In view of the foregoing, it suffices to
check $H$-equivariance of the evaluation map $M \otimes_R M^\vee
\to R$. Using the identity $\sum \sS(a_2)a_1 = \langle \e,a \rangle$ for 
$a\in H$ (from $\sS^2 = 1$), we compute
\[
a \cdot (m \otimes f) = \sum a_1m\otimes a_2f \mapsto \sum \langle a_2f,a_1m \rangle
= \sum \langle f, \sS(a_2)a_1m \rangle = \langle \e, a \rangle \langle f,m \rangle \ ,
\]
as desired.
 
\end{proof}

\subsection{} \label{SS:Hopf7}

We now focus on modules over a separable involutory Hopf algebra $H$.
In particular, we will compute the image 
of the central idempotents $e(M)$ from Proposition~\ref{P:idempotents} under the isomorphism
$\cen(H) \iso H^{\vee}_{\trace}$ in Lemma~\ref{L:Hopf} and the $\beta$-index $[H:M]_\beta$ 
of \eqref{E:index}.  Recall that $H$ is bi-symmetric by 
Lemma~\ref{L:sep}(c). 

\begin{prop} \label{P:Hopf7}
Assume that $H$ is separable and involutory. 
Fix
$\Lambda \in \int_H$\,, $\lambda \in \Hint_{H^\vee}$ such that 
$\langle \lambda, \Lambda \rangle = 1$ and let $\beta$ denote the form \eqref{E:Hopf1}\,.
Then, for every cyclic left $H$-module $M$ that is finitely generated projective over $R$
and satisfies $\End_{H}(M) \cong R$\,,
\begin{enumerate}
\item $\rank_RM$ is invertible in $R$\,;
\item $[H:M]_\beta = \langle \e,\Lambda \rangle (\rank_RM)^{-1}$ is invertible in $R$\,;
\item
$\bb_{\lambda}(e(M)) = [H:M]_\beta^{-1}\,\chi_M$\,. In particular, 
$\bb_{\lambda}(e(M)) = (\rank_RM)\,\chi_M$ holds for $\lambda = \chi_{\reg}$\,.
\end{enumerate}
\end{prop}

\begin{proof}
(a)
Our hypothesis $\End_{H}(M) \cong R$ implies that $M$ is faithful as $R$-module. Hence, the
trace map $\tr \colon \End_{R}(M) \cong M \otimes_R M^\vee \to R$ is surjective by \eqref{E:ImTr},
and it is is an $H$-module map by Lemma~\ref{L:Hopf6}. Moreover, for each $\phi \in \End_R(M)$, 
we have $\Lambda \phi = r_\phi 1_M$ for some $r_\phi \in R$, since $\Lambda \End_R(M) \subseteq \End_H(M)
\cong R$. Therefore, $\tr(\Lambda \phi) = r_\phi \rank_RM$ and  $\tr(\Lambda \phi) = \Lambda \tr(\phi)
= \langle \e,\Lambda  \rangle \tr(\phi)$\,. Choosing $\phi$ with $\tr(\phi)=1$, we obtain from 
Corollary~\ref{C:Hopfseparability} that $\tr(\Lambda \phi)$ is a unit in $R$. Hence so is $\rank_RM$,
proving (a).

(b)
By Propositions~\ref{P:idempotents}(c) and \ref{P:idempotents2}(a), we have
\[
\omega_M(z)\,\rank_RM = [H:M]_\beta\,(\rank_RM)^2\ ,
\]
where 
$z = \sum x_iy_i = \langle \e,\Lambda \rangle$ is as in the proof of Lemma~\ref{L:sep}(a).
In view on part (a), the above equality amounts to the asserted formula for $[H:M]_\beta$\,.
Finally, invertibility of $[H:M]_\beta$ is Proposition~\ref{P:idempotents}(a) (and it
also follows from Corollary~\ref{C:Hopfseparability}).

(c)
Proposition~\ref{P:idempotents2}(b) gives $\chi_{\reg} \leftharpoonup e(M) = (\rank_R M)\chi_M$\,,
which is the asserted formula for $\bb_{\lambda}(e(M))$ with $\lambda = \chi_{\reg}$\,.
For general $\lambda$, we have
$\langle \e, \Lambda \rangle \bb_{\lambda}(e(M)) = \chi_{\reg} \leftharpoonup e(M)$ by Lemma~\ref{L:sep}(a).
The formula for $\bb_{\lambda}(e(M))$ now follows from (b).
\end{proof}

\subsection{} \label{SS:weak} 

Assume that, for some subring $S \subseteq R$, there is an $S$-subalgebra $B \subseteq H$
such that 
\begin{enumerate}
\renewcommand{\labelenumi}{(\roman{enumi})}
\item 
$B$ is finitely generated as $S$-module, and
\item 
$((\sS\otimes 1_H)\circ \Delta)(\Lambda) = \sum \sS(\Lambda_1) \otimes \Lambda_2 \in H \otimes H$ 
belongs to the image
of $B \otimes_S B$ in $H \otimes H$\,.
\end{enumerate}
Adapting the teminology of Section~\ref{SS:integrality}, we will call $A$ a 
\emph{weak $R$-form} of $(H,\Lambda)$. 
The following corollary is a consequence of Proposition~\ref{P:Hopf7}(b) and Lemma~\ref{L:integrality}; it
is due to Rumynin \cite{dRxx} over fields of characteristic $0$.

\begin{cor} \label{C:HopfIrreducible}
Let $H$ be separable and involutory. 
Assume that, for some generating integral $\Lambda \in \Hint_H$\,,
there is a weak $S$-form for $(H,\Lambda)$ for some subring $S \subseteq R$.
Then, for every left $H$-module $M$ that is finitely generated projective over $R$
and satisfies $\End_{H}(M) \cong R$\,, the index  $[H:M]_\beta = \langle \e,\Lambda \rangle (\rank_RM)^{-1}$ 
is integral over $S$.
\end{cor}

\begin{example} \label{EX:grpalg}
The group algebra $RG$ of a finite group $G$ has generating (right and left) integral
$\Lambda = \sum_{g \in G} g$.
The corresponding  integral $\lambda \in \Hint_{(R G)^\vee}$ with $\langle \lambda, \Lambda \rangle = 1$ 
is the trace form given by 
$\langle \lambda,\sum_{g \in G} r_gg \rangle = r_1$.
Note that $\langle \e,\Lambda\rangle = |G|\,1$. Thus, Corollary~\ref{C:Hopfseparability} tells us
that $R G$ is semisimple if and only if
$|G|\,1$ is a unit in $R$; this is Maschke's classical theorem.
Assuming $\ch R = 0$, a weak $\ZZ$-form for $(R G, \Lambda)$ is given by the integral
group ring $B = \ZZ G$. Therefore, Corollary~\ref{C:HopfIrreducible} yields the following
version of Frobenius'
Theorem:
The rank of every $R$-free $RG$-module $M$ such that $\End_{RG}(M) \cong R$ divides the 
order of $G$.
\end{example}


\section{Grothendieck rings of semisimple Hopf algebras} \label{S:Semisimple}

From now on, we will focus on the case where $R=\k$ is a field. Throughout, we will assume
that $H$ is a split semisimple
Hopf algebra over $\k$. In particular, $H$ is finite-dimensional over $\k$ and hence Frobenius. 
We will write $\otimes = \otimes_\k$ and $\k$-linear duals will now be denoted by $(\,.\,)^*$.
Finally, $\Irr H$ will denote a full set of non-isomorphic irreducible $H$-modules.


\subsection{The Grothendieck ring} \label{SS:Grothendieck} 
We review some standard material; for details, see \cite{mL97a}.

\subsubsection{} \label{SSS:G0}
The \emph{Grothendieck ring} $G_0(H)$ of $H$ is
the abelian group that is generated by the isomorphism classes
$[V]$ of finite-dimensional left $H$-modules $V$ modulo the relations
$[V]=[U]+[W]$ for each short exact sequence $0\to U\to V\to W\to 0$.
Multiplication in $G_0(H)$ is given by $[V]\cdot [W]=[V\otimes W]$\,.
The
subset $\{ [V] \mid V \in \Irr H \} \subseteq G_0(H)$ forms 
a $\ZZ$-basis of $G_0(H)$, and the positive cone
\[
G_0(H)_+ := \bigoplus_{V \in \Irr H} \ZZ_+[V] = \{ [V] \mid \text{$V$ a finite-dimensional $H$-module}\}
\]
is closed under multiplication.
The Grothendieck ring $G_0(H)$ has the \emph{dimension augmentation},
\[
\dim \colon G_0(H) \to \ZZ \ , \qquad [V] \mapsto \dim_\k V \ ,
\]
and an involution given by $[V]^* = [V^*]$, where the dual $V^* = \Hom_\k(V,\k)$ has $H$-action
as in Section~\ref{SS:Hopf6}. The basis $\{ [V] \mid V \in \Irr H \}$ is stable under
the involution $*$\,, and hence so is the positive cone $G_0(H)_+$.

\subsubsection{} \label{SSS:G0symm}
The Grothendieck ring $G_0(H)$ is a symmetric
$*$-algebra over $\ZZ$. 
A suitable bilinear form $\beta$ 
is given by 
\begin{equation} \label{E:G0beta}
\beta([V],[W]) = \dim_\k\Hom_H(V,W^*) \ .
\end{equation}
Using the
standard isomorphism $(W \otimes V^* )^H \cong \Hom_H(V,W)$,
where $(\,.\,)^H$ denotes the space of $H$-invariants, this form is easily seen to be
$\ZZ$-bilinear, associative, symmetric, and $*$-invariant.
Dual $\ZZ$-bases of $G_0(H)$ are provided by $\{ [V] \mid V \in \Irr H \}$ and 
$\{ [V^*] \mid V \in \Irr H \}$:
\begin{equation*} \label{E:ON}
\beta([V'],[V^*]) = \delta_{[V'],[V]} \qquad (V,V' \in \Irr H) \ .
\end{equation*}

The integral in $G_0(H)$ that is associated to the dimension augmentation of $G_0(H)$
as in Section \S\ref{SS:augmentation} is the class $[H]$ of the regular representation of $H$:
$\beta([H],[V]) = \dim_\k V$. Thus,
\begin{equation} \label{E:G0int}
\Hint_{G_0(H)} = \ZZ \, [H] \ .
\end{equation}

\subsubsection{} \label{SSS:charactermap}
The \emph{character map}
\[
\chi \colon G_0(H) \to H^* \ , \qquad [V] \mapsto \chi_V
\]
is a ring homomorphism that respects augmentations:
\[
\xymatrix{%
G_0(H) \ar@{>>}[d]_{\dim} \ar[r]^-{\chi}  & H^* \ar@{>>}[d]^{u^*}   \\
\ZZ \ar[r]    &  \k}
\]
Moreover,
\[
\chi_{V^*} = \sS^*(\chi_V) = \chi_V \circ \sS \ .
\]
Thus, if $H$ is involutory then $\chi$ also commutes with the standard involutions on $G_0(H)$ 
and $H^*$.
The class of the regular representation $[H] \in \Hint_{G_0(H)}$ 
is mapped to the regular character $\chi_{\reg} \in H^*$. If $H$ is involutory then $\chi_{\reg}$ 
is a nonzero integral of $H^*$; see Lemma~\ref{L:sep}.
Thus, in this case, we have
\[
\k \,\chi(\Hint_{G_0(H)}) = \k\, \chi_{\reg} = \Hint_{H^*} \ .
\]

\subsubsection{} \label{SSS:RH}

The $\k$-algebra $R(H):= G_0(H)\otimes_{\ZZ}\k$ is called the \emph{representation algebra} of $H$.
The map $[V]\otimes 1 \mapsto \chi_V$ gives an algebra embedding $R(H) \into H^*$ whose image is
the subalgebra $H^*_{\trace} = (H/[H,H])^*$ of all trace forms on $H$:
\[
 R(H) \iso H^*_{\trace} \subseteq H^* \ .
\]


\subsection{} \label{SS:Zhu} 

As an application of Proposition~\ref{P:Hopf7}, we prove the following elegant generalization of
Frobenius' Theorem (see Example~\ref{EX:grpalg}) in characteristic $0$ 
due to S. Zhu \cite[Theorem 8]{sZ93}.

\begin{thm} \label{T:Zhu}
Let $H$ be a split semisimple Hopf algebra over a field $\k$
of characteristic $0$ and let $V \in \Irr H$ be such that $\chi_V \in \cen(H^*)$. 
Then $\dim_\k V$ divides $\dim_\k H$.
\end{thm}

\begin{proof}
By \cite[Theorem 4]{rLdR87} $H$ is involutory and cosemisimple. 
Let $\Lambda \in \Cocom(H)$ denote the character of the regular representation of $H^*$; this
is an integral of $H$ by Lemma~\ref{L:sep} and, clearly, $\langle \e,\Lambda \rangle = \dim_\k H$. 
Let $\lambda \in \Hint_{H^*}$ be such that $\langle \lambda,\Lambda \rangle = 1$ and consider
the isomorphism $\bb_{\lambda} \colon {}_HH_H \iso {}_H{H^\vee\!}_H$ in \eqref{E:Hopf4}. By
Proposition~\ref{P:Hopf7}, we have
$\bb_{\lambda}(e(V)) = \frac{\dim_\k V}{\dim_\k H}\, \chi_V$
and \eqref{E:bL} gives 
\[
\frac{\dim_\k H}{\dim_\k V}\,e(V) = \bb_{\Lambda}(\sS^*(\chi_V)) \ .
\]
Therefore, it suffices to show that $\bb_{\Lambda}(\sS^*(\chi_V))$ is integral over $\ZZ$.

By hypothesis, $\sS^*(\chi_V) \in \cen(H^*)$. Furthermore, $\sS^*(\chi_V) \in \chi(G_0(H))$ is 
integral over $\ZZ$. Hence $\sS^*(\chi_V) \in \cen(H^*)^{\text{cl}}$, the integral closure of $\ZZ$ in
$\cen(H^*)$. Passing to an algebraic closure of $\k$, as
we may, we can assume that $H^*$ and $\cen(H^*)$ are split semisimple. Thus, 
$\cen(H^*) = \bigoplus_{M\in \Irr H^*} \k e(M)$ and $\cen(H^*)^{\text{cl}} = \bigoplus_{M\in \Irr H^*} \cO e(M)$,
where we have put $\cO:= \{ \text{algebraic integers in $\k$} \}$.
Proposition~\ref{P:Hopf7}(c), with $H^*$ in place of $H$, gives
$\bb_\Lambda(e(M)) = (\dim_\k M)\chi_M$. Thus, 
\[
\bb_\Lambda(\cen(H^*)^{\text{cl}}) \subseteq \chi(G_0(H^*)) \cO \subseteq C(H) \ .
\]
Finally, all elements of $G_0(H^*)$ are integral over $\ZZ$, and hence the same holds for
the elements of $\chi(G_0(H^*)) \cO$. In particular, $\bb_{\Lambda}(\sS^*(\chi_V))$ 
is integral over $\ZZ$, as desired.
\end{proof}


\subsection{The class equation} \label{SS:classeqn} 

We now derive the celebrated class equation, due to Kac \cite[Theorem 2]{gK72}
and Zhu \cite[Theorem 1]{yZ94}, from Proposition~\ref{P:idempotents}. 
Frobenius' Theorem (Example~\ref{EX:grpalg}) in characteristic $0$ 
also follows from this result applied to $H = (\k G)^*$.
We also mentioned here that the
class equation was used by Schneider in \cite{hjS01} to prove the following strong
version of the Frobenius property for quasitriangular 
semisimple Hopf algebras $H$ over a field $\k$ of characteristic $0$:
if $V$ is an absolutely irreducible $H$-module, then $(\dim_\k V)^2$ divides $\dim_\k H$.

\begin{thm}[Class equation] \label{T:classeqn}
Let $H$ be a split semisimple Hopf algebra over a field $\k$
of characteristic $0$. Then 
$\dim_\k (H^* \otimes_{R(H)}M)$ divides $\dim_\k H^*$
for every absolutely irreducible $R(H)$-module  $M$.
\end{thm}

\begin{proof}
Inasmuch as $R(H)$ is semisimple by Corollary~\ref{C:positivity}, we have $M \cong R(H)e$ for some
idempotent $e=e^2\in R(H)$ with
$eR(H)e \cong \k$. Thus, $H^* \otimes_{R(H)}M \cong H^*e$ and the assertion of the theorem is equivalent
to the statement that $\dim_\k H^*e$ divides $\dim_\k H^*$.

The bilinear form $\beta$ in \eqref{E:G0beta} can be written as
$\beta([V],[W]) = \tau([V][W])$, where $\tau \colon G_0(H) \to \ZZ$ is the trace form given by 
\begin{equation*} \label{E:tauG0}
\tau([V]) = \dim_\k V^H \ .
\end{equation*}
Now let $\Lambda \in \Hint_H$ denote the regular character of the dual
Hopf algebra $H^*$, as in the proof of Theorem~\ref{T:Zhu}. Thus,
\begin{equation} \label{E:x}
\langle e,\Lambda \rangle \underset{\eqref{E:idempotentrank}}{=} 
\dim_\k eH^* \underset{\eqref{E:rightleft}}{=} \dim_\k H^*e \ .
\end{equation}
Being an integral of $H$, $\Lambda$ annihilates all $V \in \Irr H \setminus \{ \k_\e\}$ and so
$\langle \chi_V, \Lambda \rangle = 0$. On the other hand, 
$\langle \chi_{\k_\e}, \Lambda \rangle = \langle \e, \Lambda \rangle = \dim_\k H^*$. 
This shows that $\Lambda\big|_{R(H)} = \dim_\k H^* \cdot \tau'$, where we have put 
$\tau' = \tau \otimes_{\ZZ}\Id_\k \colon R(H) \to \k$.
Therefore, \eqref{E:x} becomes 
\begin{equation*} \label{E:x'}
\tau'(e) = \frac{\dim_\k H^*e}{\dim_\k H^*} \ .
\end{equation*}
Now, $\tau'(e) = \beta'(e,1)$, where $\beta' = \beta \otimes_{\ZZ}\Id_\k$, and by 
Proposition~\ref{P:idempotents}(a), we have 
$\beta'(e,1)^{-1} = [R(H):M]_{\beta'}$. 
Thus,
\begin{equation} \label{E:x'}
[R(H):M]_{\beta'} = \frac{\dim_\k H^*}{\dim_\k H^*e} = \frac{\dim_\k H^*}{\dim_\k (H^*\otimes_{R(H)}M)}\ .
\end{equation}
Finally, 
Lemma~\ref{L:integrality} with $A = G_0(H)$ and $A'= R(H)$
tells us that this rational number is integral over $\ZZ$. 
Hence it is an integer, proving the theorem. 
\end{proof}


\subsection{The adjoint class} \label{SS:adjoint}

\subsubsection{The adjoint representation} \label{SSS:ad}
The left adjoint representation of $H$ is given by
\begin{equation*}
\ad \colon H \tto \End_\k H\ , \quad \ad(h)(k) = \sum h_1 k \sS(h_2)
\end{equation*}
for $h, k \in H$. 
There is an $H$-isomorphism
\begin{equation} \label{E:adiso}
H_{\ad} \cong \bigoplus_{V \in \Irr H} V \otimes V^* \ .
\end{equation}
This follows from standard $H$-isomorphism $V \otimes V^* \cong \End_\k(V)$
(see Section~\ref{SS:Hopf6}) combined with the Artin-Wedderburn isomorphism, 
$H \cong \bigoplus_{V \in \Irr H} \End_\k(V)$, which is equivariant for the adjoint $H$-action
on $H$.

\subsubsection{The adjoint class} \label{SSS:adclass}
Equation~\eqref{E:adiso} gives the following description of the Casimir
element $z  = z_\beta$ of the symmetric $\ZZ$-algebra $G_0(H)$:
\begin{equation} \label{E:norm1}
z =  \sum_{V \in \Irr H} [V][V^*] = [H_{\ad}]   \in \cen(G_0(H)) \ .
\end{equation}
Therefore, we will refer to the Casimir
element $z$ as the \emph{adjoint class} of $H$. 

We now consider the left regular action of $z$ on $G_0(H)$, that is, 
the endomorphism $z_{G_0(H)} \in \End_{\ZZ}(G_0(H))$ that is given by
\[
z_{G_0(H)} \colon G_0(H) \to G_0(H)\ ,\quad x \mapsto z x \ .
\]
By Proposition~\ref{P:positivity}, we know that the eigenvalues of $z_{G_0(H)}$ are 
positive real algebraic integers and that the largest eigenvalue is $\dim(z)= \dim_\k H$.
The following proposition gives more precise information; the result was
obtained by Sommerh\"auser \cite[3.11]{yS98} using a different method.

\begin{prop} \label{P:adjoint}
Let $H$ be a split semisimple Hopf algebra over a field $\k$ of
characteristic $0$. Then 
the eigenvalues of $z_{G_0(H)}$ are positive integers $\le \dim_\k H$.
If $G_0(H)$ or $H$ is commutative then
all eigenvalues of $z_{G_0(H)}$ divide $\dim_\k H$.
\end{prop}
 
\begin{proof}
We may pass to the algebraic closure of $\k$; this changes neither $G_0(H)$ nor $z$.
Then the representation algebra
$R(H) = G_0(H)\otimes_{\ZZ}\k$ is split semisimple by Corollary~\ref{C:positivity}. 
Since $z  \in \cen(R(H))$, the eigenvalues of $z_{G_0(H)}$ are exactly the $\omega_M(z) \in \k$, 
where $M$ runs over the irreducible $R(H)$-modules and $\omega_M$ denotes the
central character of $M$ as in \ref{SS:central}. 
We know by Propositions~\ref{P:idempotents}(c) and \ref{P:idempotents2}(a) and
equation \eqref{E:x'} that
\begin{equation*} \label{E:norm2}
\omega_M(z) = \dim_\k M \cdot \frac{\dim_\k H^*}{\dim_\k (H^*\otimes_{R(H)}M)} \ ,
\end{equation*}
and this is a positive integer by Theorem~\ref{T:classeqn}.  Since $M \subseteq H^*\otimes_{R(H)}M$, we have 
$\omega_M(z) \le \dim_\k H$. If $G_0(H)$ is commutative then
$\dim_\k M=1$, and hence $\omega_M(z)$ divides $\dim_\k H$. If $H$ is commutative then $R(H) = H^*$,
and so $\omega_M(z) = \dim_\k H$. (Alternatively, if $H$ is commutative then $H_{\ad} \cong \k_\e^{\dim_\k H}$
and $z = [H_{\ad}]= (\dim_\k H)1$\,.)
\end{proof}

We mention that the Grothendieck ring $G_0(H)$ is commutative whenever the Hopf algebra
$H$ is almost commutative. In particular, this holds for all quasi-triangular
Hopf algebras; see Montgomery \cite[Section~10.1]{sM93}. 

\begin{example} \label{EX:grpalg2}
Let $H = \k G$ be the group algebra of the finite group $G$ over a splitting field $\k$ of characteristic $0$. 
The representation algebra 
$R(H) \cong \bigoplus_{V \in \Irr H} \k \chi_V \subseteq H^*$ is isomorphic to the algebra of 
$\k$-valued class functions on $G$, that is, functions $G \to \k$ that are constant on
conjugacy classes of $G$. For any finite-dimensional $\k G$-module $V$, 
the character values $\chi_V(g)$ $(g \in G)$ are the 
eigenvalues of the endomorphisms $[V]_{G_0(H)} \in \End_{\ZZ}(G_0(H))$. 
Specializing to 
the adjoint representation $V = H_{\ad}$ we obtain the eigenvalues of $z_{G_0(H)}$:
they are the integers
$\chi_{H_{\ad}}(g) = |C_G(g)|$ with $g \in G$.
\end{example}


\subsection{The semisimple locus of $G_0(H)$} \label{SS:semisimple}

Let $H$ be a split semisimple Hopf algebra over a field $\k$. 
Recall that $G_0(H)\otimes_{\ZZ}\QQ$ is semisimple by Corollary~\ref{C:positivity}. 
We will now describe the primes $p$ for which  the algebra
$G_0(H)\otimes_{\ZZ} \FF_p$ is semisimple. 

\begin{prop} \label{P:semisimple}
Let $H$ be a split semisimple Hopf algebra over a field $\k$. Then:
\begin{enumerate}
\item If $p$ divides $\dim_\k H$ then $G_0(H)\otimes_{\ZZ} \FF_p$ is not semisimple. 
\item Assume that $\ch \k =0$. Then $G_0(H)\otimes_{\ZZ} \FF_p$ is semisimple for all $p > \dim_\k H$. 
\item Assume that $\ch \k =0$ and that $G_0(H)$ or $H$ is commutative. 
Then $G_0(H)\otimes_{\ZZ} \FF_p$ is semisimple
if and only if $p$ does not divide $\dim_\k H$. 
\end{enumerate}
\end{prop}
 
\begin{proof}
Semisimplicity
is equivalent to separability over $\FF_p$; see \cite[10.7 Corollary b]{rP82}. 
Therefore, we may apply Proposition~\ref{P:separability}. 
In detail,
consider the Casimir operator that is associated with the bilinear form $\beta$ of
Section~\ref{SSS:G0symm},
\begin{equation*} \label{E:GI}
c \colon G_0(H) \to \cen(G_0(H))\ ,\quad x \mapsto \sum_{V \in \Irr H} [V^*] x [V] \ .
\end{equation*}
By Proposition~\ref{P:separability}, $G_0(H)\otimes_{\ZZ} \FF_p$ is semisimple if
and only if $(p) \nsupseteq \ZZ \cap \Im c$.

(a) Consider the dimension augmentation $\dim \colon G_0(H) \to \ZZ$, $[V] \mapsto
\dim_\k V$. The composite $\dim\circ c$ is equal to $\dim_\k H \cdot \dim$\,.
Hence, $\ZZ \cap \Im c \subseteq \Im(\dim\circ c) \subseteq (\dim_\k H)$ holds in $\ZZ$, 
which implies (a).

(b) In view of Proposition~\ref{P:adjoint}, our hypothesis on $p$ implies that the norm
$N(z) = \det z_{G_0(H)}$ is not divisible by $p$. Since $z = c(1)$, it follows that 
$(p) \nsupseteq N(\Im c)$, and hence $(p) \nsupseteq \ZZ \cap \Im c$; see Section~\ref{SS:norms}.

(c) Necessity of the condition on $p$ follows from (a) and sufficiency follows from Proposition~\ref{P:adjoint}
as in (b).
\end{proof}


\subsection{Traces of group-like elements} \label{SS:group-like}

Let $H$ be a split semisimple Hopf algebra over a field $\k$ and let $\chi_{\ad} \in R(H) \subseteq H^*$
denote the character of the adjoint representation. Equation~\eqref{E:norm1} gives
\[
\chi_{\ad} = \sum_{V \in \Irr H} \chi_{V^*}\chi_V \ .
\]

\begin{prop} \label{P:group-like}
Let $H$ be a split semisimple Hopf algebra over a field $\k$. If $R(H)$ is semisimple then
$\chi_{\ad}(g) \neq 0$ for every group-like element $g \in H$.
\end{prop}
 
\begin{proof}
By Proposition~\ref{P:separability},
semisimplicity of  $R(H)$ is equivalent to surjectivity
of the Casimir operator
$c  \colon R(H) \to \cen(R(H))$, $\chi \mapsto \sum_{V \in \Irr H} \chi_{V^*} \chi \chi_V$.
Fixing $\chi$ with $\sum_V \chi_{V^*} \chi \chi_V = 1$ we obtain
\[
1 = \sum_V \chi_{V^*}(g) \chi(g) \chi_V(g) =  \chi(g) \chi_{\ad}(g) \ ,
\]
which shows that $\chi_{\ad}(g) \neq 0$.
\end{proof}









\bibliographystyle{amsplain}
\bibliography{../bibliography}


\end{document}